\documentclass[a4paper,USenglish,cleveref,autoref]{lipics-v2021}

\pdfoutput=1
\hideLIPIcs

\usepackage{mathtools}
\usepackage{tikz}
\usetikzlibrary{positioning}
\usepackage[ruled,linesnumbered]{algorithm2e}
\nolinenumbers

\title{Concave Continuation: Linking Routing to Arbitrage}

\author{Ruichao Jiang}{Derivation Technology Ltd.}{ruichao.jiang@derivation.info}{}{}
\author{Long Wen}{Derivation Technology Ltd.}{long.wen@derivation.info}{}{}

\authorrunning{R. Jiang and L. Wen}



\begin{document}

\maketitle

\begin{abstract}
We extend AMM trade functions to negative inputs via the \textit{concave continuation}, derived from the invariance of the local conservation law under allocation direction flips. This unifies routing and arbitrage into a single problem. We extend the one-hop transfer algorithm proposed in \cite{jiang} to this setting.
\end{abstract}

\section{Introduction}

Decentralized Exchanges (DEXes) use Automated Market Makers (AMMs) \cite{univ2,univ3} to provide liquidity. When liquidity is fragmented across multiple venues, \textit{routing problem} is: Finding the best execution to maximize the output \cite{angeris-routing,diamandis-routing,jiang}. The transfer algorithm, proposed in \cite{jiang} and adopted by the Monday Trade protocol \cite{monday}, solves the one-hop routing problem by iteratively equalizing prices across pools. Despite its simplicity, the empirical study conducted by Xi and Moallemi \cite{xi-moallemi} showed that transfer algorithm would outperform many historical transactions.

On the other hand, when there are multiple AMMs between two same tokens, arbitrage opportunity may exist. The \textit{arbitrage problem} is: Use price differences to extract as much value as possible.

Routing problem is naturally related to arbitrage problem. In fact, if the allocation to some AMMs is allowed to be negative, arbitrage problem is nothing but routing problem with input size being zero.

In this article, we define \textit{concave continuation} (\Cref{def:concave-cont}) to extend AMM trade function (\Cref{def:trade-function}) to negative inputs. The extension is uniquely determined by the requirement that the local conservation law (\Cref{def:conservation-law}) be invariant when the trade direction within an AMM flips. In addition, it turns out that the concave continuation can be computed on-chain via Uniswap V3 type API.

The organization of this article is as follows. We fix the notation in \Cref{sec:background}. We present the local conservation law in AMM routing problem in \Cref{sec:conservation} and define concave continuation in \Cref{sec:concave-continuation}. In \Cref{sec:algorithm}, we propose the extended transfer algorithm and prove a theorem about its runtime behavior, from which the proofs of convergence and convergence rate from \cite{jiang,jiang-convergence} can be easily ported. In \Cref{sec:experiments}, we perform three experiments to show the efficiency of the extended transfer algorithm and that it correctly finds arbitrage opportunity.
\section{Related Work} \label{sec:related}
Angeris et al. \cite{angeris-routing} formulate optimal routing across CFMMs (Constant Function Market Makers) as a convex program solved via CVXPY.

Diamandis et al. \cite{cfmmrouter-code,diamandis-routing} solve the CFMMs routing more efficiently using L-BFGS-B (Limited-memory Broyden–Fletcher–Goldfarb–Shanno with bounds).

Diamandis et al. \cite{convex-flows} propose \textit{convex flows} over hypergraphs, unifying CFMMs routing, arbitrage, and optimal power flow.

Jiang et al. \cite{jiang} propose the one-hop transfer algorithm for routing. Xi and Moallemi \cite{xi-moallemi} propose essentially the same one-hop transfer algorithm and their empirical result  demonstrates this simple algorithm performed better than many historical on-chain swap result.

Jiang and Wen \cite{jiang-convergence} establishes an upper bound $\mathcal{O}\left(\kappa N\log\frac{1}{\varepsilon}\right)$ for the convergence rate of the one-hop transfer algorithm, where $\kappa$ is a liquidity parameter, $N$ the number of AMMs, and $\varepsilon$ some tolerance parameter. The upper bound is optimal in $N$.
\section{Preliminaries} \label{sec:background}
\begin{definition}[Trade function] \label{def:trade-function}
    The AMM's \textit{trade function}
    \begin{equation*}
        E_{X,Y}:[0,\infty)\to[0,\infty)
    \end{equation*}
    maps input $x$ of $X$ to the output of $Y$ received. The reverse trade function
    \begin{equation*}
        E_{Y,X}:[0,\infty)\to[0,\infty)
    \end{equation*}
    maps input $y$ of $Y$ to the output of $X$. We assume both are $C^1$, monotonically increasing, concave, and satisfy $E_{X,Y}(0)=E_{Y,X}(0)=0$.
\end{definition}
We assume the following conditions on $E_{X,Y}$ and $E_{Y,X}$
\begin{enumerate}
    \item $C^1$ (continuously differentiable),
    \item $E'_{X,Y}>0$, $E'_{Y,X}>0$ (monotonically increasing),
    \item Strictly concave (diminishing marginal utility),
    \item $E_{X,Y}(0)=E_{Y,X}(0)=0$ (no free lunch),
    \item $E'_{X,Y}(0)\cdot E'_{Y,X}(0)=1$ (change of num\'eraire)
\end{enumerate}
\begin{remark*}
    For Uniswap V3, if there exists no liquidity over a price range, then the trade function will not even be $C^1$. However, such case is excluded when there exists a background liquidity over the full range. Therefore, we assume that there always exists a background liquidity when dealing with Uniswap V3 pools. The $C^1$ assumption makes the presentation of this article easier as we often use first derivative as optimality condition. However, we don't assume $C^2$ differentiability because the trade function is not $C^2$ in Uniswap V3 when there are jumps in liquidity.
\end{remark*}
$E_{X,Y}'(x)$ is the \textit{price of $X$ in $Y$} when the AMM is allocated $x$ amount of $X$. Its reciprocal
\begin{equation*}
    P(x)\coloneqq\frac{1}{E_{X,Y}'(x)}
\end{equation*}
is the \textit{price of $Y$ in $X$}. The change of num\'eraire assumption says that it coincides with the first derivative of the inverse trade function.
\begin{example}[Trade functions in Uniswap V2] \label{example:v2-trade-functions}
    Uniswap V2 has trade functions
    \begin{equation*}
        E_{X,Y}(x)=\frac{r_Yx}{r_X+x}
    \end{equation*}
    and
    \begin{equation*}
        E_{Y,X}(y)=\frac{r_Xy}{r_Y+y}
    \end{equation*}
    where $r_X$ and $r_Y$ are parameters (pool's reserves).
\end{example}
We formulate the one-hop token allocation problem. A trader sells $x$ units of token $X$ for the maximum amount of token $Y$, splitting across $N$ pools:
\begin{equation} \label{eqn:nonneg}
    \begin{split}
        &\max_{\mathbf{x}}F(\mathbf{x})=\sum_{i=1}^N E_{X,Y}^i(x_i)\\
        &\text{subject to}:x_i\geq0,\ \sum_{i=1}^Nx_i=x.
    \end{split}
\end{equation}
As a convex optimization problem, Problem \eqref{eqn:nonneg} has a unique maximizer $\mathbf{x}^*$. The KKT conditions require all active pools to share a common marginal output:
\begin{equation*}
    E_{X,Y}^{'i}(x_i^*)=\lambda^*
\end{equation*}
for pools with non-zero allocation and
\begin{equation*}
    E_{X,Y}^{'i}(0)\leq\lambda^*
\end{equation*}
for inactive pools.

The transfer algorithm solves Problem \eqref{eqn:nonneg} by iteratively transferring allocation from the most overpriced (highest $P$) pool to the cheapest, using a halving rule to determine the transfer amount.
\begin{algorithm}[ht]
\caption{Transfer algorithm~\cite{jiang}}\label{alg:transfer}
\KwIn{Pools $E_{X,Y}^1,\ldots,E_{X,Y}^N$; total input $x$; tolerance $\varepsilon$}
\KwOut{Allocation $\mathbf{x}$}
Initialize \;
\While{$\frac{P_{\max}-P_{\min}}{P_{\max}}>\varepsilon$}{
    $D\gets\arg\max_i P_i(x_i)$ \tcp*{donor: most overpriced}
    $R\gets\arg\min_i P_i(x_i)$ \tcp*{receiver: cheapest}
    $\delta\gets\frac{x_D}{2}$ \tcp*{halving rule}
    \lWhile{$P_D(x_D-\delta)<P_R(x_R+\delta)$}{$\delta\gets\frac{\delta}{2}$}
    $x_D\gets x_D-\delta$ \;
    $x_R\gets x_R+\delta$ \;
}
\Return{$\mathbf{x}$}
\end{algorithm}
\section{Local Conservation Law} \label{sec:conservation}
In Problem \eqref{eqn:nonneg}, we assumed that the trading direction is from $X$ to $Y$. This is analogous to assigning the direction of electric flow when solving a circuit. Once the direction is fixed, the conservation law (Kirchhoff's law) for the circuit can be written down. However, if the solution on one edge is negative, it doesn't mean that there exists electric flow with negative amplitude. It only means we assigned a ``wrong'' direction for that edge. If we flip the preassigned direction at that edge, the solution will flip to positive on that edge. During this re-assignment, the form of the conservation law doesn't change. The preassigned direction is a redundancy in the description of the system. In physics, such redundancy is called a gauge symmetry. We will show exactly the same for AMM routing.
\begin{definition}[Local conservation law] \label{def:conservation-law}
    At a node $N$ with incoming edges carrying flows $e_1,\ldots,e_k$ and outgoing edges consuming flows $n_1,\ldots,n_m$, the \textit{local conservation law} is
    \begin{equation} \label{eqn:conservation}
        \sum_{j=1}^k e_j = \sum_{l=1}^m n_l.
    \end{equation}
\end{definition}
Consider node $Y$ (the buy token) with $N$ incoming edges from the pools and one outgoing edge to the trader. Pool~$i$ sends $E_{X,Y}^i(x_i)$ of $Y$ to the node, and the trader receives $F=\sum_i E_{X,Y}^i(x_i)$. The conservation law at $Y$ is simply
\begin{equation*}
    \sum E_{X,Y}^i(x_i)=F,
\end{equation*}
the objective of~\eqref{eqn:nonneg}. At node $X$ (the sell token), conservation gives
\begin{equation*}
    \sum x_i = X,
\end{equation*}
the budget constraint.

Now suppose that AMM $i$ flips direction: instead of receiving $Y$, it send out $Y$ via AMM $i$ to the input node.
\begin{figure}[hbt!]
    \begin{center}
    \begin{tikzpicture}[>=stealth, node distance=2.8cm, every node/.style={font=\small}]
    \node[draw, circle, minimum size=1cm] (Y) {$Y$};
    \node[left=of Y] (X) {Output};
    \node[above right=1.2cm and 2cm of Y] (P1) {Pool $1$};
    \node[right=of Y] (Pj) {Pool $j$};
    \node[below right=1.2cm and 2cm of Y] (Pn) {Pool $N$};
    \node[below left=1.2cm and 2cm of Y] (Pi) {Pool $i$};

    \draw[->] (P1) -- node[above, sloped] {$E_{X,Y}^{(1)}(x_1)$} (Y);
    \draw[->] (Pj) -- node[above] {$E_{X,Y}^{(i)}(x_i)$} (Y);
    \draw[->] (Pn) -- node[below, sloped] {$E_{X,Y}^{(N)}(x_N)$} (Y);
    \draw[->, dashed] (Y) -- node[below, sloped] {$n_i$ (flipped)} (Pi);
    \draw[->] (Y) -- node[below] {$F$} (X);
    \end{tikzpicture}
    \end{center}
    \caption{AMM $i$ flips direction}
    \label{fig:direction-flip}
\end{figure}
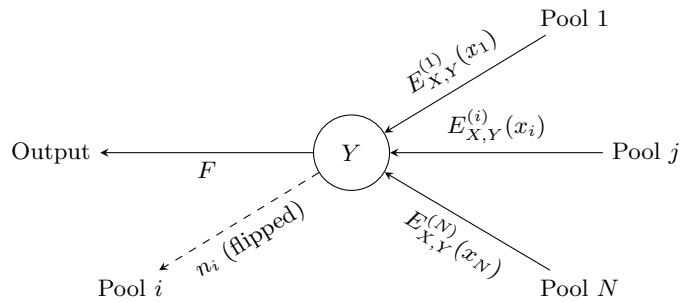

Before the flip, AMM $i$ contributes $E_{X,Y}^i(x_i)$ to the output. After the flip, AMM $i$ draws $n_i>0$ of $Y$ from the output node and returns $E_{Y,X}^i(n_i)$ of $X$ to the input node. We identify this with a negative allocation $x_i<0$ on the original assigned direction, i.e. we require:
\begin{equation*}
    E_{Y,X}^i(n_i)=-x_i.
\end{equation*}
Hence,
\begin{equation*}
    n_i=\left(E_{Y,X}^i\right)^{-1}(-x_i).
\end{equation*}
For the conservation law (Eqn \eqref{eqn:conservation}) at the output node to retain the original form:
\begin{equation*}
    \sum_i E_{X,Y}^i(x_i) = F,
\end{equation*}
we must define:
\begin{equation} \label{eqn:extension}
    E_{X,Y}^i(x_i)\coloneqq-n_i=-(E_{Y,X}^i)^{-1}(-x_i)
\end{equation}
for $x_i<0$.
\section{Concave Continuation} \label{sec:concave-continuation}
\Cref{eqn:extension} derived from the invariance of the local conservation law defines the concave continuation. In this section we establish its analytic properties and formulate the routing-arbitrage problem.

\begin{definition}[Concave continuation] \label{def:concave-cont}
    The \textit{concave continuation} of the trade function $E_{X,Y}$ to negative inputs is
    \begin{equation*}
        E_{X,Y}(x) \coloneqq -(E_{Y,X})^{-1}(-x) \quad \text{for } x<0,
    \end{equation*}
    with domain extended to $(-R_X,\infty)$, where $R_X$ is the pool's reserve of token~$X$.
\end{definition}
The term ``concave continuation'' is inspired by the analytic continuation in complex analysis. It turns out that for CFMMS, concave continuation coincides with the analytic continuation (informal \Cref{thm:analytic}).
\begin{proposition}[Properties of the concave continuation] \label{prop:properties}
    Under the assumptions in \S\ref{sec:background}, the concave continuation satisfies:
    \begin{romanenumerate}
        \item \textbf{Monotonicity.} $E_{X,Y}$ is increasing on $(-R_X,\infty)$.
        \item \textbf{Concavity.} $E_{X,Y}$ is concave on $(-R_X,\infty)$.
        \item \textbf{$C^1$ continuity at $x=0$.} $E_{X,Y}\in C^1(-R_X,\infty)$.
    \end{romanenumerate}
\end{proposition}
\begin{proof}
    Since $E_{Y,X}$ is increasing and concave, $E_{Y,X}^{-1}$ is increasing and convex for $x\geq0$. Since $E_{X,Y}(x)=-E_{Y,X}^{-1}(-x)$ is a $180^\circ$ rotation of $E_{Y,X}^{-1}$ about the origin, $E_{X,Y}(x)$ is increasing and concave for $x\in(-R_X,0]$.
    
    It remains to show that two branches of $E_{X,Y}$ glue to a single function in a $C^1$ manner.

    First,
    \begin{align*}
        \lim_{x\to0^-}E_{X,Y}(x)&=-\lim_{x\to0^-}\left(E_{Y,X}\right)^{-1}(-x)\\
        &=-\lim_{x\to0^+}\left(E_{Y,X}\right)^{-1}(x)\\
        &=0,
    \end{align*}
    where we used $E_{Y,X}(0)=0$, the no-free-lunch condition in the last line. Hence, $E_{X,Y}$ is continuous at $0$.

    Taking derivative,
    \begin{align*}
        \lim_{x\to0^-}E_{X,Y}'(x)&=\lim_{x\to0^-}\left(E'_{Y,X}\right)^{-1}(-x)\\
        &=\lim_{x\to0^+}\left(E'_{Y,X}\right)^{-1}(x)\\
        &=\lim_{x\to0^+}\frac{1}{E_{Y,X}'(x)}\\
        &= \lim_{x\to0^+}E_{X,Y}'(x),
    \end{align*}
    where we used the inverse function theorem in the second last line and the change of num\'eraire assumption in the last line. Hence, $E'_{X,Y}$ is continuous at $0$.
\end{proof}
\begin{example}[Concave continuation for Uniswap V2] \label{example:v2-concave-continuation}
    The inverse function of $E_{Y,X}(y)$ in \Cref{example:v2-trade-functions} is
    \begin{equation*}
        E_{Y,X}^{-1}(x)=\frac{r_Yx}{r_X-x},
    \end{equation*}
    $x\in[0,r_X)$.
    Thus, the concave continuation of $E_{X,Y}$ is
    \begin{align*}
        E_{X,Y}(x)&=-(E_{Y,X})^{-1}(-x)\\
        &=-\frac{r_Y(-x)}{r_X+x}\\
        &=\frac{r_Yx}{r_X+x},
    \end{align*}
    which is identical in functional form with $E_{X,Y}$ for $x\geq0$.
\end{example}
\Cref{example:v2-concave-continuation} is not a coincidence and has nothing to do with the fact that Uniswap V2's trade functions are symmetric under $x\leftrightarrow y$ and $r_x\leftrightarrow r_y$. We formulate the following informal theorem.
\begin{theorem}[Informal] \label{thm:analytic}
    If $E_{X,Y}$ and $E_{Y,X}$ are defined implicitly by an \textit{invariant curve} $G:[0,\infty)\times[0,\infty)\to\mathbb{R}$ via
    \begin{equation*}
        G(r_x+x,r_y-E_{X,Y}(x))=G(r_x,r_y)
    \end{equation*}
    and
    \begin{equation*}
        G(r_x-E_{Y,X}(y),r_y+y)=G(r_x,r_y)
    \end{equation*}
    and that $G(r_x,r_y)=0$ defines a real analytic function $r_y(r_x)$, then the concave continuation of $E_{X,Y}$ coincides with the analytic continuation of $E_{X,Y}$.
\end{theorem}
\begin{remark*}
     We leave out the definition of invariant curve as it is irrelevant to the transfer algorithm. Also, \Cref{thm:analytic} doesn't apply to Uniswap V3.
\end{remark*}
\begin{remark*}
    We outline a proof of \Cref{thm:analytic}: First, establish that $E_{X,Y}$ and $-(E_{Y,X})^{-1}$ have the same derivatives of all orders at point $0$ so that they have identical Taylor expansions. Then, invoke the identity theorem in complex analysis to conclude $E_{X,Y}\equiv-(E_{Y,X})^{-1}$.
\end{remark*}
The following example shows that \Cref{thm:analytic} doesn't apply to Uniswap V3, which is the focus of transfer algorithm.
\begin{example}[Concentrated liquidity at a tick boundary] \label{ex:tick-boundary}
    Consider a Uniswap~V3 pool sitting at a tick boundary where liquidity jumps. The forward direction ($x\geq0$) enters a tick with reserves $(1,1)$, and the reverse direction ($y\geq0$) enters a tick with reserves $(2,2)$:
    \begin{align*}
        E_{X,Y}(x)&=\frac{x}{1+x},\\
        E_{Y,X}(y)&=\frac{2y}{2+y}.
    \end{align*}
    Both quote the same price at the boundary: $E_{X,Y}'(0)=\frac{1}{E_{Y,X}'(0)}=1$. The concave continuation of $E_{X,Y}$ for $x<0$ is
    \begin{equation*}
        -(E_{Y,X})^{-1}(-x) = \frac{2x}{2+x},
    \end{equation*}
    which is a \emph{different formula} from $\frac{x}{1+x}$. The extended function is $C^1$ at $x=0$ (prices match) but not $C^2$: the second derivative jumps from $E_{X,Y}''(0^+)=-2$ to $E_{X,Y}''(0^-)=-1$, reflecting the liquidity change across the tick. 
\end{example}
With the concave continuation, routing and arbitrage are unified:
\begin{equation} \label{eqn:extended}
    \begin{split}
        &\max_{\mathbf{x}}F(\mathbf{x})=\sum_{i=1}^N E_{X,Y}^i(x_i)\\
        &\text{subject to}:x_i\geq -R_i,\ \sum_{i=1}^N x_i = X,
    \end{split}
\end{equation}
where $R_i$ is the reserve of token $X$ in $i$-th pool. Pools with $x_i>0$ produce $Y$; pools with $x_i<0$ consume $Y$. The objective is strictly concave (\Cref{prop:properties}), so the optimum is unique.
\begin{remark*}
    $X=0$ corresponds to a pure arbitrage.
\end{remark*}
\begin{remark*}
    Problem \eqref{eqn:extended} has $N$ variables, $1$ linear equality, and $N$ box constraints. The nonlinearity resides in the concave objective. This is significant easier than introducing dual variables for each edge because that introduces non-linear terms to the constraint.
\end{remark*}

\section{The Extended Transfer Algorithm} \label{sec:algorithm}
In this section, we extend the transfer algorithm to the domain $x_i\in(-R_i,\infty)$. We list the algorithms first.
\Cref{alg:init} initializes by greedily distributing $X$ in portions to the currently cheapest pool.
\begin{algorithm}[htbp!]
    \caption{Greedy initialization} \label{alg:init}
    \KwIn{Pools $E_{X,Y}^1,\ldots,E_{X,Y}^N$; total input $X$; number of portions $M$}
    \KwOut{Initial allocation $\mathbf{x}^{(0)}$}
    $\mathbf{x}^{(0)}\gets\mathbf{0}$ \;
    \For{$i=1,\ldots,M$}{
        $R\gets\arg\min_i P_i\left(x_i^{(0)}\right)$ \;
        $x_R^{(0)}\gets x_R^{(0)}+\frac{X}{M}$ \;
    }
    \Return{$\mathbf{x}^{(0)}$}
\end{algorithm}

\Cref{alg:halving} gives the extended halving rule.
\begin{algorithm}[htbp!]
    \caption{Extended halving rule} \label{alg:halving}
    \KwIn{Donor $D$, receiver $R$, allocations $x_D$ and $x_R$, reserve $R_D$}
    \KwOut{Legitimate transfer amount $\delta$}
    \eIf{$x_D > 0$}{
        $\delta\gets\frac{x_D}{2}$ \tcp*{bisect $[0, x_D]$; donor stays $\geq 0$}
    }{
        $\delta\gets\frac{R_D+x_D}{2}$ \tcp*{bisect $[-R_D, x_D]$; donor stays in domain}
    }
    \lWhile{$P_D(x_D - \delta) < P_R(x_R + \delta)$}{
        $\delta\gets\frac{\delta}{2}$
    }
    \Return{$\delta$}
\end{algorithm}

\Cref{alg:extended} is the extended transfer algorithm.
\begin{algorithm}[htbp!]
\caption{Extended transfer algorithm}\label{alg:extended}
\KwIn{Pools $E_{X,Y}^1,\ldots,E_{X,Y}^N$ with reserves $R_1,\ldots,R_N$; total input $X$; tolerance $\varepsilon$}
\KwOut{Allocation $\mathbf{x}$}
Greedy initialization by \Cref{alg:init} \;
\While{$\frac{P_D-P_R}{P_D}>\varepsilon$}{
    Compute $\delta$ by extended halving rule (\Cref{alg:halving}) \;
    $x_D\gets x_D-\delta$ \;
    $x_R\gets x_R+\delta$ \;
    $D,R\gets\arg\max_i P_i(x_i),\arg\min_i P_i(x_i)$ \;
}
\Return{$\mathbf{x}$}
\end{algorithm}

Then we prove some results about the runtime of \Cref{alg:extended}, from which both the convergence and convergence rate of can be obtained following \cite{jiang,jiang-convergence}.

Denote by $S$ the set of all pools, $\mathbf{x}^*$ the optimal allocation of problem \eqref{eqn:extended} with common marginal $\lambda^*$. Define the optimal allocation sets:
\begin{align*}
    &S_+\coloneqq\{j\mid x_j^*>0\},\\
    &S_-\coloneqq\{j\mid x_j^*<0\},\\
    &S_0\coloneqq S-S_+-S_-.
\end{align*}
We use the superscript $k$ to indicate those sets during the execution of \Cref{alg:extended}.
Now we derive stronger results that hold throughout the execution of the algorithm.
\begin{theorem}[Attractors] \label{theorem-attractors}
    Under the legitimacy requirement, the following holds.
    \begin{romanenumerate}
        \item If $j\in S^{(l)}_+$ for some round $l\geq0$, $j\in S^{(m)}_+$ for all $m>l$.
        \item If $j\in S^{(l)}_-$ for some round $l\geq0$, $j\in S^{(m)}_-$ for all $m>l$.
    \end{romanenumerate}
\end{theorem}
\begin{proof}
    (i) Suppose for contradiction that $j\in S-S^{(m)}_+$ for some $m>l$. WLOG, we may assume $j$ were chosen as a donor in round $m$, $x^{(m)}_j>0$ and $\delta\geq x^{(m)}_j$. Denote by $R^{(m)}$ the receiver in round $m$.
    
    Since $j$ acquired its positive allocation only by becoming a receiver in some round $k<l$ with $x^{(k)}_j=0$, it would follow that
    \begin{align*}
        \min_{i\in S}P_i\left(x_i^{(k)}\right)&=P_j(0)\geq P_j\left(x_j^{(m)}-\delta\right)\geq P_{R^{(m)}}\left(x_R^{(m)}+\delta\right)\\
        &>P_{R^{(m)}}\left(x_R^{(m)}\right)=\min_{i\in S}P_i\left(x_i^{(m)}\right)\geq\min_{i\in S}P_i\left(x_i^{(k)}\right),
    \end{align*}
    where the first and third inequalities are by the convexity of $P$, the second by the legitimacy requirement, and the last by $P_{\min}$ being non-decreasing in round, a contradiction.
    
    (ii) Suppose for contradiction that $j\in S-S^{(m)}_-$ for some $m>l$. WLOG, we may assume $j$ were chosen as a receiver in round $m$, $x^{(m)}_j<0$ and $\delta\geq-x^{(m)}_j$. Denote by $D^{(m)}$ the donor in round $m$.
    
    Since $j$ acquired its negative allocation only by becoming a donor in some round $k<l$ with $x^{(k)}_j=0$, it would follow that
    \begin{align*}
        \max_{i\in S}P_i\left(x_i^{(k)}\right)&=P_j(0)\leq P_j\left(x_j^{(m)}+\delta\right)\leq P_{D^{(m)}}\left(x_{D}^{(m)}-\delta\right)\\
        &<P_{D^{(m)}}\left(x_{D}^{(m)}\right)=\max_{i\in S}P_i\left(x_i^{(m)}\right)\leq\max_{i\in S}P_i\left(x_j^{(k)}\right),
    \end{align*}
    a contradiction.
\end{proof}
\begin{corollary}
    Under the legitimacy requirement, the following holds.
    \begin{romanenumerate}
        \item If $j\in S_+$, then $j\notin S^{(k)}_-$ for all $k\geq0$.
        \item If $j\in S_-$, then $j\notin S^{(k)}_+$ for all $k\geq0$.
    \end{romanenumerate}
\end{corollary}
In summary, the runtime of \Cref{alg:extended} exhibits the following.
\begin{romanenumerate}
    \item $S^{(k)}_+$ expands to $S_+$ and once a pool joins $S^{(k)}_+$, it should (by legitimacy requirement, which implies convergence of algorithm) and will (by \Cref{alg:halving}) stay in it.
    \item $S^{(k)}_-$ expands to $S_-$ and once a pool joins $S^{(k)}_-$, it should and will stay in it.
    \item $S^{(k)}_0$ shrinks to $S_0$ and once a pool leaves $S^{(k)}_0$, it shouldn't and won't return to it.
\end{romanenumerate}
It may happen that $S_0\neq\emptyset$. The members in it are those whose initial price happens to be the optimal price.
\begin{theorem}[Convergence] \label{thm:convergence}
    \Cref{alg:extended} converges to the unique optimal of problem \eqref{eqn:extended}.
\end{theorem}
\begin{proof}
    Under the legitimacy requirement, $P_D-P_R$ is always non-increasing. We need to show that it always decreases\footnote{Except in a rather trivial case: There are two pool, one with price same as donor and the other as receiver. Then after a legitimate transfer, $P_D-P_R$ stays the same. This can be eliminated by the next round and isn't an obstacle to convergence.}. The only way that it doesn't decrease is a pool with positive/negative optimal allocation is allocated with a negative/positive amount during the run of the algorithm so that \Cref{alg:halving} can't correct it to the optimal sign. However, \Cref{theorem-attractors} and its corollary say that the legitimacy requirement also implies that such scenario may not occur.
\end{proof}
\begin{remark*}[Convergence rate]
    The proof of the convergence rate $\mathcal{O}\left(N\kappa\log\frac{1}{\varepsilon}\right)$ established in \cite{jiang-convergence} also hold. The essence is that every finite size legitimate transfer makes non-infinitesimal improvement to the objective function. And this is upheld if a pool always stays in the region where its optimal allocation lies.
\end{remark*}
\section{Experiments} \label{sec:experiments}
We validated the extended transfer algorithm (\Cref{alg:extended}) by three experiments: a Python implementation against \cite{angeris-routing}, a Julia implementation against \cite{angeris-routing,diamandis-routing}, and an on-chain Solidity implementation on Base mainnet against the original transfer algorithm \cite{jiang}. We omitted a comparison against \cite{convex-flows} as it turned out hard to use their formulation in this setting without implementing some feathers ourselves.

\subsection{Python simulation on V2 pools} \label{sec:exp-python}

As \cite{angeris-routing} is implemented in Python and only supports Uniswap V2 pools, we compare our algorithm with them in Julia and using Uniswap V3. In addition, \cite{angeris-routing} supports arbitrage by using two non-negative variables $(\Delta_i,\Lambda_i)$ for each edge.

We sampled $100$ Uniswap V2 pools and the input was $X=100$. We used the original transfer algorithm \cite{jiang} as a baseline comparison for output.
\begin{table}[htbp!]
    \centering
    \begin{tabular}{r | r r | r r | r r}
        \hline
        & \multicolumn{2}{c|}{\textbf{Transfer} (\Cref{alg:extended})} & \multicolumn{2}{c|}{CVXPY \cite{angeris-routing}} & \multicolumn{2}{c}{Transfer \cite{jiang}} \\
        $N$ & ms & $Y$ & ms & $Y$ & ms & $Y$ \\
        \hline
        3   &  0.21 &  58.87 &     12.1 &  58.87 & \textbf{0.32} &  58.87 \\
        5   &  0.27 & 114.50 &     17.7 & 114.50 & \textbf{0.20} & 114.12 \\
        10  &  0.65 & 146.13 &     31.9 & 146.13 & \textbf{0.21} & 132.56 \\
        20  &  1.29 & 284.32 &     61.1 & 284.32 & \textbf{0.21} & 175.69 \\
        50  &  3.27 & 550.76 &    149.7 & 550.76 & \textbf{0.21} & 208.45 \\
        100 &  6.79 & 946.28 &    313.3 & 946.28 & \textbf{0.21} & 234.36 \\
        \hline
    \end{tabular}
    \caption{Routing+arbitrage on randomly generated V2 pools.}
    \label{tab:python}
\end{table}

\Cref{alg:extended} and \cite{angeris-routing} agreed on output at every $N$, but is 8--22 times faster. The arbitrage gain over routing only became significant when $N\geq 10$. The baseline was very fast and we found that it was caused by: After the greedy initialization, most pools' prices were above the highest post-allocation price (When $N=100$, only $15$ pools were active).

\subsection{Julia simulation on V3 pools} \label{sec:exp-sim}
As \cite{diamandis-routing} is implemented in Julia and supports Uniswap V3 pools, we compare our algorithm with them in Julia and using Uniswap V3.

We sampled $100$ Uniswap V3 pools and the input was $X=100$. We used the original transfer algorithm \cite{jiang} as a baseline comparison for output.

\begin{table}[htbp!]
    \centering
    \begin{tabular}{r | r r | r r | r r}
        \hline
        & \multicolumn{2}{c|}{\textbf{Extended transfer} (ours)} & \multicolumn{2}{c|}{CFMMRouter~\cite{diamandis-routing}} & \multicolumn{2}{c}{Transfer~\cite{jiang}} \\
        $N$ & $\mu$s & $Y$ & $\mu$s & $Y$ & $\mu$s & $Y$ \\
        \hline
        5   & 48  &  98.55 &     121 &  98.55 &    68 &  98.55 \\
        10  &           76 & 128.47 &     112 & 128.47 &    33 & 127.47 \\
        20  & 60  & 188.37 &     195 & 188.37 &    44 & 163.47 \\
        50  &          650 & 306.91 &     515 & 306.92 &    80 & 184.72 \\
        100 & 461 & 521.79 &   1{,}052 & 521.82 &   136 & 204.54 \\
        \hline
    \end{tabular}
    \caption{Routing+arbitrage on randomly generated V3 pools.}
\label{tab:julia}
\end{table}

\Cref{alg:extended} and \cite{diamandis-routing} agreed on output at every $N$ but was always faster. The non-arbitrage baseline \cite{jiang} produced strictly less output except when $N=5$. However, the baseline was slower only when $N=5$.

\subsection{On-chain validation on Base} \label{sec:exp-onchain}
We test on Base mainnet (block~43{,}550{,}927) using 4 Uniswap V3 USDC/WETH pools at fee tiers 1bp, 5bp, 30bp, and 100bp.

The extended algorithm (\Cref{alg:extended}) can be run fully on-chain with realistic gas cost and no additional smart contract. The post-allocation price is calculated using \texttt{sqrtPriceX96After} returned by Uniswap V3 quoters. The concave continuation is calculated on-chain by \texttt{quoteExactOutputSingle}, also native to the Uniswap V3.

\subparagraph*{No arbitrage opportunity.}
All 4 pools shows no sign of arbitrage opportunity.  We traded USDC$\to$WETH with input amounts from \$10 to \$100{,}000. Concave extension produced no additional output than the original algorithm. However, the gas cost was also the same.

There are two reasons. First, USDC/WETH pair is among the most traded pairs on-chain and MEV bots can capture any arbitrage opportunity very quickly. Second, real AMMs always have fees, which make a small price discrepancy a non-profitable ``arbitrage''.
\begin{table}[htbp!]
    \centering
    \label{tab:equilibrium}
    \begin{tabular}{r | r r r}
        \hline
        Input (USDC) & WETH output & Gas & Rounds \\
        \hline
        10       & 0.00455 & 1.4M & 0 \\
        100      & 0.04549 & 1.5M & 0 \\
        1{,}000  & 0.4549  & 1.7M & 0 \\
        10{,}000 & 4.5449  & 5.3M & 2 \\
        100{,}000 & 45.470  & 40M  & 8 \\
        \hline
    \end{tabular}
    \caption{No arbitrage opportunity within four Uniswap V3 USDC/WETH pools on Base.}
\end{table}

\subparagraph*{Induced price discrepancy.}
To create an arbitrage opportunity and demonstrate that the extended algorithm also works in real-world setting, we inserted a swap of size 100{,}000~USDC into the 100bp pool before the run of \Cref{alg:extended}. We then routed in a fresh 100{,}000~USDC. The result is in \Cref{tab:induced-price-discrepancy}.
\begin{table}[htbp!]
    \centering
    \begin{tabular}{l | r r}
        \hline
        & WETH output & Gas \\
        \hline
        Transfer algorithm \cite{jiang} & 45.121 & 7.5M \\
        \textbf{Concave extension} & \textbf{45.344} & \textbf{3.2M} \\
        \hline
        Improvement & \textbf{+49 bps} ($\approx$ \$560) & \\
        \hline
    \end{tabular}
    \caption{Arbitrage opportunity after induced price discrepancy.}
    \label{tab:induced-price-discrepancy}
\end{table}

The original algorithm avoided the overpriced pool entirely and routed through the remaining three. The concave extension additionally assigned the inflated pool a negative allocation: sending WETH back to extract more USDC, and routed the freed USDC through cheaper pools for additional output. We note that the gas cost was lowered. However, we doubt if there is any theoretical guarantee for that.
\section{Discussion} \label{sec:discussion}
In this article, we proposed concave continuation, which is derived from the local conservation law in routing problem. We unified the arbitrage problem with the routing problem via the concave continuation.

We extended the transfer algorithm in \cite{jiang} using concave continuation. We proved a runtime property of the extended transfer algorithm (\Cref{theorem-attractors}), which shows that $0$ as an impenetrable barrier of allocation amounts during the run of \Cref{alg:extended}. This prevents a pool with positive/negative optimal allocation to stuck with the wrong allocation under the halving rule (\Cref{alg:halving}). Equivalently, it says during the allocation process, a pool may not acquire an allocation of the opposite sign to that of the optimal one during the entire run of \Cref{alg:extended}.

The experiment result showed the expected behavior and efficiency. In fact, the extended transfer algorithm can be implemented and run on-chain with reasonable gas cost.

Theoretically, concave continuation is the first step toward a multi-hop transfer algorithm: If there are some intermediate tokens $M$ and $N$ in the trade from token $X$ to $Y$, we do not know a priori what the trade direction between $M$ and $N$ should be nor are we certain the allocation will flip signs or not during the routing algorithm's run. This remains future work.

\bibliography{biblio}

\end{document}